\documentclass[english
]{article}
\usepackage[T1]{fontenc}
\usepackage[latin9]{inputenc}
\usepackage{amssymb, amsmath, amsthm}
\usepackage{babel}

\usepackage{amsmath, amsthm, amscd, graphicx,amssymb,comment,
cite
}

\def\D{\Delta}

\def\lam{\lambda}

\newcommand{\Id}{\operatorname{Id}\nolimits}
\newcommand{\Ad}{\operatorname{Ad}\nolimits}
\newcommand{\spann}{\operatorname{span}\nolimits}
\newcommand{\SOLV}{\operatorname{SOLV}\nolimits}

\renewcommand{\Vec}{\operatorname{Vec}\nolimits}

\def\vH{\overrightarrow{H}}

\newcommand{\be}[1]{\begin{equation}\label{#1}}
\newcommand{\ee}{\end{equation}}
\newcommand{\eq}[1]{$(\protect\ref{#1})$}
\newcommand{\map}[3]{#1 \, : \, #2 \to #3}

\def\R{{\mathbb R}}

\def\kinv{{\kappa}}
\def\xinv{{\chi}}

\newcommand{\heis}{\operatorname{h}_3\nolimits}
\newcommand{\SO}{\operatorname{SO(3)}\nolimits}
\newcommand{\SE}{\operatorname{SE(2)}\nolimits}
\newcommand{\SU}{\operatorname{SU(2)}\nolimits}
\newcommand{\SH}{\operatorname{SH(2)}\nolimits}
\newcommand{\SL}{\operatorname{SL(2)}\nolimits}

\newcommand{\so}{\operatorname{so(3)}\nolimits}
\newcommand{\sh}{\operatorname{sh(2)}\nolimits}
\newcommand{\se}{\operatorname{se(2)}\nolimits}
\renewcommand{\sl}{\operatorname{sl(2)}\nolimits}

\renewcommand{\phi}{\varphi}

\def\ds{\displaystyle}

\newtheorem{theorem}{Theorem}
\newtheorem{lemma}{Lemma}

\newcommand{\rank}{\operatorname{rank}}

\begin{document}

\title{Superintegrability of  Sub-Riemannian Problems on Unimodular 3D Lie Groups\footnote{Work supported by 
Grant of the Russian Federation for the State Support of Researches
(Agreement  No~14.B25.31.0029).}}

\author{Alexey P. Mashtakov,  Yuri L. Sachkov}

\maketitle

\begin{abstract}
Left-invariant sub-Riemannian problems on unimodular 3D Lie groups are considered. For the Hamiltonian system of Pontryagin maximum principle for sub-Riemannian geodesics, the Liouville integrability and superintegrability are proved. 
\end{abstract}




\section{Introduction}
Let $G$ be a connected 3-dimensional Lie group, and $L$ the Lie algebra of left-invariant vector fields on $G$. A left-invariant contact sub-Riemannian structure on $G$ is a rank 2, left-invariant subbundle $\D \subset TG$, $\D + [\D,\D]=TG$, endowed with a left-invariant inner product $g$ in $\D$. 
Sub-Riemannian (SR) minimizers  are   Lipschitzian curves $q:[0,t_1] \to G$ such that $\dot{q}(t) \in \D_{q(t)}$ for almost all $t \in [0,t_1]$, and the length of the curve 
$\ds l(q(\cdot)) = \int_{0}^{t_1} \sqrt{g(\dot{q}(t),\dot{q}(t))} \, dt$
is the minimum possible for all curves that connect two given points: $q(0) = q_0$ and  $q(t_1)=q_1$. 

A left-invariant sub-Riemannian structure can be defined by a left-invariant orthonormal frame $f_1, f_2 \in L$:
\be{frame}
\D_q = \spann (f_1(q),f_2(q)), \qquad g(f_i(q),f_j(q)) = \delta_{ij}, \quad i,j = 1,2. 
\ee
Then SR  minimizers are solutions to the optimal control problem
\begin{eqnarray}
&&\dot q = u_1 f_1(q) + u_2 f_2(q), \qquad q \in G, \qquad (u_1, u_2) \in \R^2, \label{pr1}\\
&&q(0) = q_0, \qquad q(t_1) = q_1, \label{pr2}\\
&&l = \int_{0}^{t_1} \sqrt{u_1^2 + u_2^2} \, dt \to \min. \label{pr3} 
\end{eqnarray}
SR geodesics  are curves in $G$ whose sufficiently short arcs are SR  minimizers.

Sub-Riemannian geometry is a rapidly developing domain of mathematics at the crossroads of differential geometry, PDEs, optimal control and calculus of variations, metric analysis, Lie groups and Lie algebras theory, and other important domains, with rich applications to classical and quantum mechanics, robotics, neurophysiology and vision, etc~\cite{mont, notes, ABB, stricharts, gromov,versh_gersh, brock, ledonne}.

In this work we are interested in the problem of describing SR minimizers. The most efficient approach to this problem is based on optimal control theory \cite{PBGM,  jurd_book, notes}, and it consists of the following steps: 
\begin{enumerate}
\item \label{enum:st1} proof of existence of SR   minimizers, which is a standard corollary of the Rashevsky-Chow and Filippov theorems,
\item \label{enum:st2}
parametrization of SR geodesics via Pontryagin maximum principle,
\item \label{enum:st3}
selection of SR  minimizers among SR geodesics via second order optimality conditions and detailed study of structure of the family of SR geodesics.
\end{enumerate} 
Along this sequence, complexity of the problems grows exponentially. Existence of SR length minimizers is standard for left-invariant problems on Lie groups. 
Explicit parameterization of SR geodesics was performed in many problems. Complete description of all SR length minimizers (optimal synthesis in optimal control problem~\eq{pr1}--\eq{pr3}) is known just in several simplest cases: the Heisenberg group \cite{versh_gersh, brock}, $\SO$, $\SU$, $\SL$ with the Killing metric \cite{boscain_rossi}, $\SE$ \cite{cut_sre2},  
on the Engel group~\cite{engel, engel_conj, engel_cut}, for 2-step corank 2 nilpotent SR problems~\cite{BBG_step2}. Finding a parameterization of SR geodesics can be a nontrivial problem even for left-invariant SR structures on Lie groups. 
So a natural question arises on a theoretical possibility of such parameterization in some reasonable sense --- the question of integrability of ODEs that determine the SR geodesics. 
In paper \cite{shapiro} was constructed an example of a 6-dimensional Lie group with a nonintegrable ODE for SR geodesics.

In this paper we prove that the Hamiltonian system of ODEs for SR geodesics is Liouville integrable (and superintegrable) for any contact left-invariant SR structure on a 3-dimensional unimodular Lie group. That is, we consider Lie groups $G$ with the Lie algebras $L=\heis$, $\so$, $\sl$, $\se$, and $\sh$. 

Recall that a Hamiltonian vector field $\vH$ on a symplectic manifold $M$, $\dim M = 2 d$, is called  Liouville integrable if it has $d$ independent integrals in involution, i.e., there exist functions $f_1 = H, f_2, \dots, f_d \in C^{\infty}(M)$ such that $\{f_i, f_j\} = 0$, $i, j = 1, \dots, d$, and $f_1$, \dots, $f_d$ are functionally independent on an open dense subset of $M$~\cite{arnold_mech}. A Hamiltonian vector field $\vH$ is called superintegrable (or integrable in noncommutative sense)~\cite{nehoroshev, mish_fom} if there exist functions $f = (f_1 = H, f_2, \dots, f_{2d-n})$, $f_i \in C^{\infty}(M)$, functionally independent on an open dense subset of $M$, such that:
\begin{eqnarray}
&&\{f_i, f_j\} = P_{ij} \circ f, \qquad \map{P_{ij}}{f(M)}{\R}, \quad i, j = 1, \dots, 2d-n,
\label{fifjPij}\\
&&
\text{the matrix  $(P_{ij})$ has rank  $2d - 2n$ on an open dense subset of $f(M)$}.
\label{rankPij}
\end{eqnarray}
In the case $n = d$ superintegrability reduces to Liouville integrability. A review of geometry of superintegrable system may be found e.g. in~\cite{fasso}.

\section{Sub-Riemannian structures \\and Pontryagin maximum principle}
Left-invariant contact SR structures on 3D Lie groups $G$ were classified up to local isometries in a recent work by A.~Agrachev and D.~Barilari \cite{agrachev_barilari}. In particular, it was shown that if $G$ is unimodular, i.e.,  its Lie algebra $L$ is one of the Lie algebras $\heis$, $\so$, $\sl$, $\se$, or $\sh$, then there exists an orthonormal frame 
$L = \spann(f_0,f_1,f_2)$ such that $f_1$, $f_2$ satisfy~\eq{frame} and
\be{fifj}
\left[f_2,f_1\right] = f_0, \qquad 
\left[f_1, f_0\right] = (\xinv + \kinv)f_2, \qquad
\left[f_2,f_0\right] = (\xinv - \kinv) f_1,
\ee
for some constants $\xinv \geq 0$ and $\kinv \in \R$.

It is well known~\cite{notes} that arclength parameterized SR geodesics for contact left-invariant problems on Lie groups are projections 
$q(t) = \pi(\lambda(t))$, $\pi: T^{\ast}G \to G$,
of trajectories of the Hamiltonian system
$\dot{\lambda} = \overrightarrow{H}(\lambda)$, $\lambda \in T^{\ast}G$,
with the Hamiltonian function
$H(\lambda) = \frac12(h_1^2(\lambda)+h_2^2(\lambda))$,
$h_i(\lambda) = \langle\lambda, f_i(q)\rangle$, $q = \pi(\lambda)$.
Here the Hamiltonian vector field $\overrightarrow{H}$ on the cotangent bundle $T^{\ast}G$ is defined by the equality
$\sigma_{\lambda}(\cdot,\overrightarrow{H}) = d_{\lambda}H$, $\lambda \in T^{\ast}G$,
where $\sigma = ds$, $s_{\lambda}= \lambda \circ \pi_{\ast}.$
The aim of this work is the proof of  integrability of the Hamiltonian vector field $\overrightarrow{H}$ in the Liouville and noncommutative sense.

By virtue of the Lie brackets (\ref{fifj}), we have the Poisson brackets
$\left\{H,h_1\right\}    = h_2 h_0$,
$\left\{H, h_2\right\}    = -h_1 h_0$,
$\left\{H,h_0\right\}  =   2 \xinv h_1 h_2$.
Thus the Hamiltonian system  
$\dot{\lambda} = \overrightarrow{H}(\lambda)$ reads as follows:
\begin{eqnarray}
&\dot{h}_1 = h_2 h_0, \qquad 
\dot{h}_2 = -h_1 h_0, \qquad 
\dot{h}_0 = 2 \xinv h_1 h_2, \label{vert}\\
&\dot{q} = h_1 f_1 + h_2 f_2. \nonumber
\end{eqnarray} 
In the polar coordinates 
$h_1 = r \cos \theta$, $h_2 = r \sin \theta$, 
the vertical subsystem~\eq{vert} reduces to 
$$
\dot{r} = 0,\qquad 
\dot{\theta} = -h_0, \qquad
\dot{h}_0 = \xinv r^2 \sin 2 \theta.
$$ 
Further, in the coordinates
$ \gamma = 2 \theta$, $c = -2 h_0$
we get the equation of pendulum
\be{pend}
\dot{r} = 0, \qquad 
\dot{\gamma} = c, \qquad 
\dot{c} = - 2 \xinv r^2 \sin \gamma.
\ee
This equation has the integral of full energy 
$$E = \frac{c^2}{2} - 2 \xinv r^2 \cos \gamma = 2 h_0^2 - 2 \xinv(h_1^2-h_2^2).$$
Thus the Hamiltonian vector field $\vH$ has two left-invariant integrals: the Hamiltonian $H$ and the energy $E$ of pendulum~\eq{pend}.

\section{Right-invariant Hamiltonians} \label{sec:Ham}
For any right-invariant vector field $e \in \Vec(G)$, one can define the corresponding Hamiltonian $g_e(\lambda) = \langle\lambda,e\rangle$,  $\lambda \in T^{\ast} G$.
Since right translations commute with the left ones,  left-invariant vector fields commute with   right-invariant ones. Thus left-invariant Hamiltonians Poisson-commute with the right-invariant ones. So right-invariant Hamiltonians provide a natural source of integrals for left-invariant Hamiltonian vector fields (in particular, for left-invariant optimal control problems).

A standard procedure to construct a right-invariant vector field $e \in \Vec(G)$ from a left-invariant one is to apply the inversion $i:G \to G$, $i(q)=q^{-1}$. If $f \in \Vec(G)$ is left-invariant, then $e(\Id) = -f(\Id)$. This construction preserves Lie brackets: if $[f_i,f_j]=\sum_k c_{ij}^k f_k$, then $[e_i,e_j]=\sum_k c_{ij}^k e_k$.

Consider the right-invariant frame $e_1$, $e_2$, $e_0 \in \Vec(G)$ constructed thus from the left-invariant frame $f_1$, $f_2$, $f_0 \in \Vec(G)$. We have a decomposition
$e_i= a_i^0 f_0 + a_i^1 f_1 + a_i^2 f_2$, $a_i^j \in C^{\infty}(G)$,
with $a_i^j(\Id) = -\delta_i^j, \, i,j = 0,1,2$.
Then the right-invariant Hamiltonians $g_i(\lambda) = \langle\lambda, e_i\rangle$ admit the decomposition 
$g_i= a_i^0 h_0 + a_i^1 h_1 + a_i^2 h_2$.

The Hamiltonian vector field $\overrightarrow{H}$ has   integrals
$H$, $E$, $g_0$, $g_1$, $g_2$,
with the only nonzero Poisson brackets following from the multiplication table~(\ref{fifj}):
\be{gigj}
\left\{g_2, g_1\right\}  = g_0, \qquad 
\left\{g_1, g_0\right\}  = (\xinv + \kinv)g_2, \qquad 
\left\{g_2, g_0\right\}  = (\xinv - \kinv) g_1. 
\ee 

\section{Liouville integrability}
In order to study the functional independence of the integrals $H$, $E$, $g_0$, $g_1$, $g_2$ at a point $\lambda \in T^{\ast} G$, $\pi(\lambda) = \Id$, introduce the Jacobian matrix
$$J(\lambda) = (\nabla H \, \frac14 \nabla E \, \nabla g_0 \, \nabla g_1 \, \nabla g_2)^T.$$
Let $(x_0,x_1,x_2)$ be local coordinates in a neighborhood of $\Id \in G$, then we obtain in the coordinates $\left(h_0, h_1, h_2; x_0, x_1, x_2\right)$ on $T^{\ast}G$:
\begin{equation}
\label{JId}
J = \left(\begin{array}[c]{c c c c c c}
0 & h_1 & h_2 & 0 & 0 & 0 \\	
h_0 & -\xinv h_1 & \xinv h_2 & 0 & 0 & 0 \\
-1 & 0 & 0 & g_{00} & g_{01} & g_{02} \\
0 & -1 & 0 & g_{10} & g_{11} & g_{12} \\
0 & 0 & -1 & g_{20} & g_{21} & g_{22} 
\end{array}
\right),
\end{equation}
where $g_{ij} = \ds\frac{\partial g_i}{\partial x_j}(\Id).$
Liouville integrability of the field $\overrightarrow{H}$ follows by the study of the vertical derivatives of the integrals (i.e., the derivatives w.r.t. the variables $h_i$).
\begin{theorem}
\label{th1}
\begin{itemize}
\item[$(1)$]
The Hamiltonian vector field $\overrightarrow{H}$ has integrals
\begin{equation}
\label{integrals}
H, \ E, \ g_0, \ g_1, \ g_2
\end{equation}
with the nonzero Poisson brackets~\eq{gigj}.
\item[$(2)$]
Integrals~\eq{integrals} are functionally dependent since
\be{depend}
4 \kinv H + E = 2(g_0^2 + (\kinv - \xinv)g_1^2 +(\kinv+\xinv)g_2^2).
\ee
The functions in the left-hand side and right-hand side of identity~\eq{depend} are Casimir functions on $L^*$.
\item[$(3)$] The field $\overrightarrow{H}$ is Liouville integrable. Specifically:
\begin{enumerate}
\item[$(3.1)$] If $\xinv \neq 0$, then for any $g= \alpha_0 g_0 + \alpha_1 g_1 + \alpha_2 g_2$, $(\alpha_0, \alpha_1, \alpha_2) \in \R^3 \backslash \{0\}$, the integrals  $H$, $E$, $g$ are in involution and are functionally independent on an open dense subset of $T^{\ast}G$;
\item[$(3.2)$] If $\xinv = 0$, then the same property holds for any $g= \alpha_0 g_0 + \alpha_1 g_1 + \alpha_2 g_2$, with  $\alpha_1^2 + \alpha_2^2 \neq 0$. 
\end{enumerate}
\end{itemize}
\end{theorem}
\begin{proof}
Item $(1)$ was proved in Section \ref{sec:Ham}. 

Now we prove item $(2)$. Consider the left-invariant Hamiltonian
$$
C_l = 4 \kinv H + E =2(h_0^2 + (\kinv - \xinv)h_1^2 +(\kinv+\xinv)h_2^2)
$$
and its right-invariant counterpart
$$
C_r = 2(g_0^2 + (\kinv - \xinv)g_1^2 +(\kinv+\xinv)g_2^2),
$$
then the required identity~\eq{depend} reads $C_l = C_r$. One checks immediately that
$$
\{C_l,h_i\} = \{C_l,g_i\} = \{C_r,h_i\} = \{C_r,g_i\} =0, \qquad i = 0, 1, 2,   
$$
thus $C_l$ and $C_r$ are Casimir functions. Now we prove that they coincide one with another.

If $\lam \in T_{\Id}^*G$, then
$
h_j(\lam) = \langle \lam, f_j(\Id)\rangle  = \langle \lam, -e_j(\Id)\rangle = - g_j(\lam)$,
thus $C_l(\lam) = C_r(\lam)$.

Now take arbitrary $q \in G$ and $\lam \in T_q^*G$. Then
\begin{eqnarray}
C_l(\lam) &=& C_l(L_q^* \lam) = C_r(L_q^* \lam) = C_r(\Ad_{q^{-1}}^* L_q^* \lam) = C_r(R_q^* L_{q^{-1}}^* L_q^* \lam)= C_r(R_q^* \lam) 
\nonumber\\
&=& C_r(\lam),
\label{chain}
\end{eqnarray}
and the identity $C_l=C_r$ is proved. In the proof of chain~\eq{chain} we used the left-invariant property of $C_l$ (1-st equality), the inclusion $L_q^*\lam \in T_{\Id}^* L$ (2-nd equality), the fact that Casimir functions are constant on co-adjoint orbits~\cite{kirillov}
$$
\{\Ad_{q^{-1}}^*(\lam) \mid q \in G\}, \qquad  \Ad_{q^{-1}}^*(\lam) = R_q^* L_{q^{-1}}^* \lam
$$
(3-rd equality), and the right-invariant property of $C_r$ (6-th equality).

In order to prove item $(3)$, notice that   by virtue of analyticity, functional independence of integrals on an open dense domain in $T^{\ast}G$ follows from linear independence of gradients of integrals at a single point $\lambda \in T^{\ast}G$. Let $\pi(\lambda) = \Id$. Then the equality
\begin{eqnarray*}
\frac{\partial(H,E,g)}{\partial (h_0,h_1,h_2)}(\lambda) = \begin{vmatrix} 0 & h_1 & h_2  \\	
h_0 & -\xinv h_1 & \xinv h_2  \\
\alpha_0 & \alpha_1 & \alpha_2 \end{vmatrix}
= 2 \xinv \alpha_0 h_1 h_2 + \alpha_1 h_0 h_2 - \alpha_0 h_0 h_1 \neq 0
\end{eqnarray*}
implies item $(3)$.
\end{proof}

\section{Superintegrability}
There arises a natural question on the number of functionally independent integrals (\ref{integrals}). We  answer  this question after computing the derivatives $g_{ij}$ in (\ref{JId}). We will do this for a special class of local coordinates on $G$ considered e.g. in \cite{bellaiche, jean}.  

A system of local coordinates $(x_0, \dots , x_n)$ on a smooth manifold $M$ is called linearly adapted to a frame $f_0, \dots, f_n \in \Vec(M)$ at a point $q \in M$ if 
$\frac{\partial}{\partial x_i}(q) = f_i(q)$, $i=0,\dots,n$.
If $f_0, \dots, f_n$ is a left-invariant frame on a Lie group $G$, then both canonical coordinates of the first kind
$(x_0, \dots , x_n) \mapsto e^{x_0 f_0 + \dots + x_n f_n}$
and canonical coordinates of the second kind 
$(x_0, \dots , x_n) \mapsto e^{x_n f_n}\dots e^{x_0 f_0}$
are linearly adapted to the frame $f_0, \dots, f_n$ at the identity $\Id \in G$.

The coefficients $g_{ij}$ can be computed by the following general proposition.

\begin{lemma}
Let $G$ be a Lie group, $f_0, \dots, f_n \in \Vec(G)$ a left-invariant vector frame, $e_j = i_{\ast}f_j$ the corresponding right-invariant vector fields. Let 
\begin{equation}
\label{eifj}
e_i = \sum_{j=0}^n a_i^j f_j, \quad a_i^j \in C^{\infty}(G).
\end{equation}
If coordinates $(x_0, \dots, x_n)$ are linearly adapted to the frame $f_0, \dots, f_n$ at the identity $\Id \in G$, then
\begin{equation}
\label{daikdxj}
\frac{\partial a_i^k}{\partial x_j} (\Id) = c_{ji}^k,
\end{equation}
where $[f_i, f_j] = \sum_{k=0}^n c_{ij}^k f_k$, $\ i,j,k = 0,\dots, n$.
\end{lemma}
\begin{proof}
Introduce the dual coframe on $G$:
$\omega_0, \dots, \omega_n \in \Lambda^1(G)$, $\langle\omega_i, f_j\rangle = \delta_{ij}$, $i,j = 0,\dots, n$.
Then decomposition (\ref{eifj}) reads $a_i^j = \langle\omega_j, e_i\rangle$. 

Further, for any function $\varphi \in C^{\infty}(G)$ we have $\frac{\partial \varphi}{\partial x_j}(\Id) = (f_j \varphi)(\Id)$. Thus
$\frac{\partial a_i^k}{\partial x_j}(\Id) = (f_j a_i^k)(\Id)$, $i,j,k = 0,\dots,n$.
Now we compute the derivative in the right-hand side via Leibnitz's rule:
$$
Y \langle\omega, X\rangle = \langle L_Y \omega, X\rangle + \langle\omega, [Y,X]\rangle, \qquad X,Y \in \Vec(G), \quad \omega \in \Lambda^1(G).
$$
Since left-invariant fields commute with right-invariant ones, we have
$$f_j a_i^k = f_j \langle \omega_k, e_i\rangle = \langle L_{f_j}\omega_k, e_i\rangle + \langle\omega_k, [f_j,e_i]\rangle = \langle L_{f_j}\omega_k, e_i\rangle.$$
On the other hand,
\begin{eqnarray*}
0 &=& f_j \delta_{ki} = f_j \langle\omega_k, f_i\rangle = \langle L_{f_j}\omega_k, f_i\rangle + \langle\omega_k, [f_j,f_i]\rangle=\\
&=& \langle L_{f_j}\omega_k, f_i\rangle + \langle\omega_k, \sum_{l=0}^n c_{ji}^l f_l\rangle = \langle L_{f_j}\omega_k, f_i\rangle + c_{ji}^k.
\end{eqnarray*}
Thus
$$(f_j a_j^k)(\Id) = \langle L_{f_j}\omega_k, e_i\rangle(\Id) = -\langle L_{f_j}\omega_k, f_i\rangle(\Id) = c_{ji}^k,$$
and equality (\ref{daikdxj}) follows.
\end{proof}
Now we can compute the coefficients 
$$g_{ij} = \frac{\partial g_i}{\partial x_j}(\Id) = \sum_{k=0}^2 \frac{\partial a_i^k}{\partial x_j}(\Id) h_k = \sum_{k=0}^2 c_{ji}^k h_k,$$
thus
\begin{equation}
\label{J}
J = \left(\begin{array}[c]{c c c c c c}
0 & h_1 & h_2 & 0 & 0 & 0 \\	
h_0 & -\xinv h_1 & \xinv h_2 & 0 & 0 & 0 \\
-1 & 0 & 0 & 0 & (\xinv+\kinv) h_2 & (\xinv-\kinv) h_1 \\
0 & -1 & 0 & -(\xinv+\kinv) h_2 & 0 & h_0 \\
0 & 0 & -1 & (\kinv-\xinv) h_1 & -h_0 & 0 
\end{array}
\right).
\end{equation}

Since the integrals~\eq{integrals} are dependent by item~(2) of Th.~\ref{th1}, then $\rank J <5$. On the other hand,  
\begin{eqnarray*}
\frac{\partial(H,g_0,g_1,g_2)}{\partial (h_0,h_1,h_2,x_1)} = \begin{vmatrix} 0 & h_1 & h_2 & 0 \\	
-1 & 0 & 0 & (\xinv + \kinv) h_2  \\
0 & -1 & 0 & 0 \\ 0 & 0 & -1 & -h_0\end{vmatrix}
= - h_0 h_1 \not\equiv 0.
\end{eqnarray*}
Thus if $\pi(\lambda) = \Id$ and $h_0(\lam) h_1(\lam) \neq 0$, then $\rank J(\lambda)=4$.
We get the following statement.

\begin{theorem}
The integrals $H$, $g_0$, $g_1$, $g_2$ are functionally independent on the set $\{\lam \in T^*G \mid h_0(\lam) h_1(\lam) \neq 0\}$. 

The Hamiltonian vector field $\vH$ is superintegrable.
\end{theorem}
\begin{proof}
Functional independence of the 4 integrals was proved before statement of this theorem.

Conditions of superintegrability~\eq{fifjPij}, \eq{rankPij} with $d = 3$, $n = 2$ follow immediately from the  Poisson brackets~\eq{gigj}.
\end{proof}

\section{Conclusion}
In this paper we proved that the Hamiltonian system of ODEs for left-invariant SR
geodesics is Liouville integrable (and superintegrable) for any
 3-dimensional unimodular Lie group, namely, for Lie groups with the Lie algebras
$\heis$, $\so$, $\sl$, $\se$, and $\sh$.   
Explicit integration of these Hamiltonian systems was done in other works for the following
 Lie groups (in terms of the invariants $\xinv$, $\kinv$ in~\eq{fifj}):
\begin{itemize}
\item[$\xinv = \kinv = 0$]: the Heisenberg group~\cite{versh_gersh, brock},
\item[$\xinv = 0$, $\kinv \neq 0$]: the Lie groups $\SO$ and $\SL$ with the Killing metric~\cite{boscain_rossi},
\item[$\xinv = \kinv \neq 0$]: the group of Euclidean motions of the plane  $\SE$~\cite{max_sre, cut_sre1, cut_sre2},
\item[$\xinv = - \kinv \neq 0$]: the group of hyperbolic motions of the plane  $\SH$~\cite{sh2, mazhitova}.
\item[$0 < \xinv < \kinv$]: the Lie group $\SO$   with general left-invariant SR metric~\cite{bonnard_cots_sher}.
\end{itemize} 
The remaining cases ($\xinv > \max(0, \kinv)$, $\xinv \neq \kinv$, i.e., the Lie group $\SL$   with general left-invariant SR metric) are still to be studied. We hope that results of this paper can be useful for the analysis of these cases. Moreover, even in the cases already integrated, the Liouville integrability and superintegrability of the Hamiltonian system may provide a geometric information about dynamics of the Hamiltonian flow~\cite{arnold_mech, fasso}.

An interesting direction of further study of integrability of SR problems is suggested by the recent classification of left-invariant SR structures of Engel type on 4D Lie groups and homogeneous spaces~\cite{almeida}. 

\section*{Acknowledgment}
We thank Dr. Lev Lokutsievskiy for indicating the short proof of item (2) of Th.~\ref{th1}.


\end{document}